\documentclass[11pt]{article}
\usepackage[margin=1in]{geometry}
\usepackage{amsmath,amssymb,amsthm}
\usepackage{microtype}
\usepackage{xcolor}
\usepackage[colorlinks=true,linkcolor=blue!50!black,citecolor=blue!50!black,urlcolor=blue!50!black]{hyperref}
\hypersetup{pdftitle={Almost-sure uniqueness of the Gaussian location NPMLE},pdfauthor={Haiyang Wang}}
\newtheorem{theorem}{Theorem}
\newtheorem{proposition}{Proposition}
\newtheorem{lemma}{Lemma}
\theoremstyle{remark}
\newtheorem{remark}{Remark}
\newcommand{\R}{\mathbb R}
\newcommand{\Pcal}{\mathcal P}
\newcommand{\supp}{\operatorname{supp}}
\newcommand{\conv}{\operatorname{conv}}
\title{Almost-sure uniqueness of the Gaussian location NPMLE}
\author{Haiyang Wang\thanks{Department of Applied and Computational Mathematics, Yale University, New Haven, CT 06511, USA (e-mail: \href{mailto:haiyang.wang1024@gmail.com}{haiyang.wang1024@gmail.com}).}}
\date{September 20, 2026}

\begin{document}
\maketitle

\begin{abstract}
For Gaussian location mixtures with identity covariance, we prove that the nonparametric maximum likelihood estimator of the mixing distribution is unique for Lebesgue-almost every dataset, for every sample size and dimension. In particular, uniqueness holds almost surely whenever the joint distribution of the observations is absolutely continuous. The proof uses the finite-support theorem for Gaussian location NPMLEs and almost-everywhere differentiability of the optimal log-likelihood. Wherever this optimal value is differentiable, all NPMLEs have the same density derivatives at the observations. A minimal linear dependence among Gaussian evaluation vectors then rules out distinct solutions. An appendix extends the result to arbitrary fixed, known, observation-specific positive-definite covariances.
\end{abstract}

\paragraph{Note on preparation.}
This work was developed and drafted with GPT-6 Astra Ultra. The author has checked the mathematical arguments for correctness and made edits to improve the exposition. The exposition has not yet been fully refined in this preprint; further details appear in the declaration of AI use below.

\section{Introduction}

Let $\Pcal(\R^d)$ denote the Borel probability measures on $\R^d$. For a mixing distribution $\pi\in\Pcal(\R^d)$, write
\[
p_\pi(x)=\int\varphi_d(x-\theta)\,d\pi(\theta),\qquad \varphi_d(x)=(2\pi)^{-d/2}e^{-\|x\|^2/2},
\]
for its Gaussian location mixture density. Given observations $X=(x_1,\ldots,x_n)\in(\R^d)^n$, consider the log-likelihood optimization problem
\begin{equation}\label{eq:npmle}
\sup_{\pi\in\Pcal(\R^d)}\ell_X(\pi),\qquad \ell_X(\pi):=\sum_{i=1}^n\log p_\pi(x_i).
\end{equation}
A nonparametric maximum likelihood estimator (NPMLE) for $X$, denoted by $\widehat\pi$, is any mixing measure attaining this supremum. We denote the optimal log-likelihood by $\ell^\star(X):=\sup_{\pi\in\Pcal(\R^d)}\ell_X(\pi)=\ell_X(\widehat\pi)$ and write $\conv(X):=\conv\{x_1,\ldots,x_n\}$. Throughout, \emph{almost every} dataset means every $X$ outside a set of Lebesgue measure zero in $\R^{nd}$.

For univariate Gaussian location mixtures with fixed variance, the NPMLE is unique for every dataset~\cite{lindsay1993}. In higher dimensions, uniqueness can fail even for isotropic Gaussian kernels~\cite[Lemma~2]{soloff2025}. Wang~\cite[Theorem~2]{wang2026} proved that every Gaussian location NPMLE is finitely supported and asked whether the multivariate NPMLE is almost surely unique under Gaussian-mixture sampling~\cite[Section~4]{wang2026}. Theorem~\ref{thm:uniqueness} answers this question by showing that all datasets admitting multiple NPMLEs lie in a Lebesgue-null set.

We collect the known properties used in the proof below.

\begin{proposition}[Known NPMLE properties]\label{prop:known}
For every dataset $X$, the following statements hold.
\begin{enumerate}
\renewcommand{\labelenumi}{(\roman{enumi})}
\item There exists $\widehat\pi$ solving~\eqref{eq:npmle} with at most $n$ atoms. Problem~\eqref{eq:npmle} has a unique likelihood vector, the same for every solution $\widehat\pi$~\cite[Lemma~1]{soloff2025}:
\[
\widehat p(X):=\bigl(p_{\widehat\pi}(x_i)\bigr)_{i=1}^n\in(0,\infty)^n.
\]
\item For every $\widehat\pi$ solving~\eqref{eq:npmle}, $\supp(\widehat\pi)\subseteq\conv(X)$~\cite[Lemma~3 and Proposition~4(1)]{soloff2025}, and $\supp(\widehat\pi)$ is finite~\cite[Theorem~2]{wang2026}.
\item The dual certificate satisfies
\begin{equation}\label{eq:certificate}
D_X(\theta):=\frac1n\sum_{i=1}^n\frac{\varphi_d(x_i-\theta)}{\widehat p_i(X)}\leq1,\qquad \supp\widehat\pi\subseteq\{\theta:D_X(\theta)=1\},
\end{equation}
for every $\widehat\pi$ solving~\eqref{eq:npmle}~\cite[Lemma~1]{soloff2025}.
\end{enumerate}
The solution $\widehat\pi$ need not be unique: already for $n=3$ and $d=2$, a suitably scaled equilateral triangle admits a continuum of NPMLEs~\cite[Lemma~2]{soloff2025}.
\end{proposition}

\begin{theorem}\label{thm:uniqueness}
For every $n,d\geq1$, the NPMLE $\widehat\pi$ is unique for almost every $X\in\R^{nd}$.
\end{theorem}

\begin{remark}[Number of atoms]
Since there always exists a solution with at most $n$ atoms~\cite[Lemma~1]{soloff2025}, uniqueness implies $|\supp(\widehat\pi)|\leq n$. Thus this bound holds for the unique NPMLE at almost every dataset in Theorem~\ref{thm:uniqueness}.
\end{remark}

Theorem~\ref{thm:uniqueness} gives almost-sure uniqueness whenever the observations have a jointly absolutely continuous distribution. This includes independent observations from any Gaussian location mixture.

The result also extends to the heteroscedastic model of Soloff, Guntuboyina, and Sen~\cite{soloff2025}: for every fixed tuple of known positive-definite covariance matrices, the NPMLE is unique for almost every dataset. Appendix~\ref{app:heteroscedastic} gives the precise statement and proof.

The proof combines finite support with almost-everywhere differentiability of $\ell^\star$. At datasets where $\ell^\star$ is differentiable, equality of the NPMLEs' density derivatives rules out a minimal linear dependence among their Gaussian evaluation vectors. Cox~\cite[Section~2, Lemma~1]{cox2020} developed general uniqueness criteria based on derivatives with respect to random inputs. Here, convexity and Rademacher's theorem provide the needed differentiability in Lemma~\ref{lem:differentiability}, and Lemma~\ref{lem:derivatives} gives equality of derivatives.

\section{Proof of the main theorem}
\begin{lemma}[Differentiability of the optimal log-likelihood]\label{lem:differentiability}
The function $\ell^\star:\R^{nd}\to\R$ is locally Lipschitz and hence differentiable almost everywhere.
\end{lemma}
\begin{proof}
Define the auxiliary function
\[
F(X):=\ell^\star(X)+\frac{nd}{2}\log(2\pi)+\frac12\sum_i\|x_i\|^2.
\]
Since $\varphi_d\leq(2\pi)^{-d/2}$ and $\pi$ is a probability measure, $p_\pi(x_i)\leq(2\pi)^{-d/2}$ for every $\pi$. Taking logarithms, summing, and taking the supremum gives $\ell^\star(X)\leq-\frac{nd}{2}\log(2\pi)$. This upper bound, together with the lower bound obtained by choosing $\pi=\delta_0$, yields
\[
0\leq F(X)\leq\frac12\sum_i\|x_i\|^2.
\]
Expanding the Gaussian exponent in~\eqref{eq:npmle} yields
\[
F(X)=\sup_{\pi\in\Pcal(\R^d)}\sum_i\log\int e^{x_i\cdot\theta-\|\theta\|^2/2}\,d\pi(\theta).
\]
Each logarithmic integral is convex by H\"older's inequality, and taking the supremum preserves convexity. Thus $F$ is finite and convex, and therefore locally Lipschitz~\cite[Example~9.14]{rockafellarwets1998}. Subtracting the smooth quadratic shows that $\ell^\star$ is locally Lipschitz. Rademacher's theorem~\cite[Theorem~9.60]{rockafellarwets1998} states that every locally Lipschitz function on an open subset of Euclidean space is differentiable outside a set of Lebesgue measure zero; applying it to $\ell^\star$ proves the claim.
\end{proof}

\begin{lemma}[Equality of likelihood derivatives]\label{lem:derivatives}
At every dataset $X$ where $\ell^\star$ is differentiable, every $\widehat\pi$ solving~\eqref{eq:npmle} satisfies
\begin{equation}\label{eq:derivatives}
\nabla p_{\widehat\pi}(x_i)=\widehat p_i(X)\,\nabla_{x_i}\ell^\star(X),\qquad i=1,\ldots,n.
\end{equation}
In particular, these derivatives do not depend on the choice of $\widehat\pi$.
\end{lemma}
\begin{proof}
By Proposition~\ref{prop:known}, $\widehat\pi$ has compact support, so $X'\mapsto\ell_{X'}(\widehat\pi)$ is smooth. The definition of $\ell^\star$ gives $\ell^\star(X')-\ell_{X'}(\widehat\pi)\geq0$, with equality at $X'=X$. Differentiating at this minimum, while keeping $\widehat\pi$ fixed, yields
\[
\nabla_{x_i}\ell^\star(X)=\nabla\log p_{\widehat\pi}(x_i)=\frac{\nabla p_{\widehat\pi}(x_i)}{\widehat p_i(X)},
\]
where the last equality uses the unique likelihood vector from Proposition~\ref{prop:known}. This proves~\eqref{eq:derivatives}.
\end{proof}

\begin{proof}[Proof of Theorem~\ref{thm:uniqueness}]
By Lemma~\ref{lem:differentiability}, it suffices to prove uniqueness at a dataset $X$ where $\ell^\star$ is differentiable. Write $\widehat p=\widehat p(X)$ and suppose two distinct solutions $\widehat\pi_0,\widehat\pi_1$ exist. By Proposition~\ref{prop:known}, both are finitely supported. Enumerate their combined support by distinct points $\theta_1,\ldots,\theta_m$, let $u,v\in\R^m$ be their weight vectors, and put $w=(u+v)/2$. Define the $n\times m$ Gaussian evaluation matrix $\Phi$ by $\Phi_{ij}=\varphi_d(x_i-\theta_j)$. Uniqueness of the likelihood vector in Proposition~\ref{prop:known} gives
\[
w_j>0,\qquad \Phi u=\Phi v=\Phi w=\widehat p,\qquad 0\ne u-v\in\ker\Phi.
\]
Every $\theta_j$ belongs to $\supp(\widehat\pi_0)\cup\supp(\widehat\pi_1)$, so~\eqref{eq:certificate} gives $D_X(\theta_j)=1$. With $q_i=(n\widehat p_i)^{-1}$, this implies
\begin{equation}\label{eq:normalization}
q^\top\Phi=\mathbf1^\top,\qquad \Phi h=0\ \Longrightarrow\ \mathbf1^\top h=0.
\end{equation}

Choose $0\ne h\in\ker\Phi$ with \emph{inclusion-minimal support}: writing $T:=\{j:h_j\ne0\}$, there is no nonzero $z\in\ker\Phi$ whose support is a proper subset of $T$. Such a choice exists because $\ker\Phi\ne\{0\}$ and there are only finitely many subsets of $\{1,\ldots,m\}$. Let $\Phi_T$ be the submatrix containing the columns indexed by $T$, and let $h_T=(h_j)_{j\in T}$. Then
\begin{equation}\label{eq:circuit}
\ker\Phi_T=\operatorname{span}\{h_T\},\qquad h_j\ne0\quad(j\in T).
\end{equation}
Indeed, if $a\in\ker\Phi_T$ were not a scalar multiple of $h_T$, then for any $j\in T$ the vector $b:=a-(a_j/h_j)h_T$ would be nonzero, belong to $\ker\Phi_T$, and satisfy $b_j=0$. Extending $b$ by zeros outside $T$ would give a nonzero vector in $\ker\Phi$ supported on a proper subset of $T$, a contradiction. Also $|T|\geq2$, since every entry of $\Phi$ is positive.

For sufficiently small $\varepsilon>0$, the vectors $w\pm\varepsilon h$ are positive, sum to one by~\eqref{eq:normalization}, and have likelihood vector $\widehat p$. They therefore define two solutions $\widehat\pi_+$ and $\widehat\pi_-$. Set
\[
g(x):=\sum_{j\in T}h_j\varphi_d(x-\theta_j)=\frac{p_{\widehat\pi_+}(x)-p_{\widehat\pi_-}(x)}{2\varepsilon}.
\]
Since $\Phi h=0$, we have $g(x_i)=0$; Lemma~\ref{lem:derivatives} also gives $\nabla g(x_i)=0$. For each $k=1,\ldots,d$, the Gaussian identity $\partial_{x_k}\varphi_d(x-\theta)=(\theta_k-x_k)\varphi_d(x-\theta)$ yields
\[
0=\partial_kg(x_i)+x_{i,k}g(x_i)=\sum_{j\in T}\Phi_{ij}h_j\theta_{j,k},\qquad\text{hence}\qquad (h_j\theta_{j,k})_{j\in T}\in\ker\Phi_T.
\]
By~\eqref{eq:circuit}, for each $k$ there is a scalar $c_k$ such that $h_j\theta_{j,k}=c_kh_j$ for every $j\in T$. Dividing by $h_j\ne0$ shows that all $\theta_j$, $j\in T$, coincide, contradicting their distinctness and $|T|\geq2$. Uniqueness therefore holds outside the Lebesgue-null exceptional set in Lemma~\ref{lem:differentiability}.
\end{proof}

\section*{Declaration of AI use}

AI assistance was used in developing and drafting the proofs of Theorems~\ref{thm:uniqueness} and~\ref{thm:hetero-uniqueness}, as well as in preparing the exposition of this paper. In particular, the mathematical arguments and manuscript were developed through the author's interactions with GPT-6 Astra Ultra. The author has checked the mathematical arguments and statements for correctness and made edits to the exposition, and takes full responsibility for the contents of the paper. The exposition has not yet been fully refined and may be improved in subsequent versions of this preprint.

\appendix
\numberwithin{theorem}{section}
\numberwithin{proposition}{section}
\numberwithin{lemma}{section}
\numberwithin{equation}{section}
\section{Heteroscedastic Gaussian location mixtures}\label{app:heteroscedastic}

We use the model of Soloff, Guntuboyina, and Sen~\cite{soloff2025}, in which observation $i$ has a known positive-definite covariance matrix $\Sigma_i$. Fix $\Sigma=(\Sigma_1,\ldots,\Sigma_n)$ and define
\[
\varphi_{\Sigma_i}(z):=\frac{\exp(-\tfrac12 z^\top\Sigma_i^{-1}z)}{(2\pi)^{d/2}(\det\Sigma_i)^{1/2}},\qquad p_{\pi,i}(x):=\int\varphi_{\Sigma_i}(x-\theta)\,d\pi(\theta).
\]
The corresponding optimization problem and optimal value are
\begin{equation}\label{eq:hetero-npmle}
\sup_{\pi\in\Pcal(\R^d)}\ell_{X,\Sigma}(\pi),\qquad \ell_{X,\Sigma}(\pi):=\sum_{i=1}^n\log p_{\pi,i}(x_i),\qquad \ell^\star_\Sigma(X):=\sup_{\pi\in\Pcal(\R^d)}\ell_{X,\Sigma}(\pi).
\end{equation}
An NPMLE $\widehat\pi$ is any mixing measure attaining the supremum in~\eqref{eq:hetero-npmle}, with $X$ and $\Sigma$ understood from context.

\begin{theorem}\label{thm:hetero-uniqueness}
For every $n,d\geq1$ and every fixed tuple $\Sigma=(\Sigma_1,\ldots,\Sigma_n)$ of positive-definite covariance matrices, there is a Lebesgue-null set $N_\Sigma\subset\R^{nd}$ such that, for every $X\notin N_\Sigma$, problem~\eqref{eq:hetero-npmle} has a unique solution $\widehat\pi$, and $|\supp(\widehat\pi)|\leq n$.
\end{theorem}

\paragraph{Adaptation of the earlier proof.}
The covariance tuple is held fixed throughout; the exceptional set $N_\Sigma$ may depend on it. The proof of Theorem~\ref{thm:uniqueness} changes in three places: Proposition~\ref{prop:hetero-known} supplies finite support, the convex auxiliary function uses $\Sigma_i^{-1}$ in its quadratic terms, and the Gaussian derivative identity is applied separately to each observation's kernel. The minimal linear-dependence argument is unchanged.

\subsection*{Proof of Theorem~\ref{thm:hetero-uniqueness}}

\begin{proposition}[Known heteroscedastic NPMLE properties]\label{prop:hetero-known}
For every $X$ and every fixed positive-definite covariance tuple $\Sigma$, there exists $\widehat\pi$ solving~\eqref{eq:hetero-npmle} with at most $n$ atoms. The likelihood vector $\widehat p=(p_{\widehat\pi,i}(x_i))_{i=1}^n$ is the same for every solution, and
\begin{equation}\label{eq:hetero-certificate}
D_{X,\Sigma}(\theta):=\frac1n\sum_{i=1}^n\frac{\varphi_{\Sigma_i}(x_i-\theta)}{\widehat p_i}\leq1,\qquad \supp(\widehat\pi)\subseteq\{\theta:D_{X,\Sigma}(\theta)=1\}
\end{equation}
for every solution $\widehat\pi$~\cite[Lemma~1]{soloff2025}. Moreover, every such $\widehat\pi$ is finitely supported; see~\cite[Corollary~6.2]{amendola2026} and~\cite[Theorem~1]{okuno2026}.
\end{proposition}

The finite-support conclusion also follows from Okuno's finiteness theorem for the critical points of a finite Gaussian mixture~\cite[Theorem~1]{okuno2026}: $D_{X,\Sigma}$ is a positive multiple of such a mixture, and every point where $D_{X,\Sigma}=1$ is a global maximum and hence a critical point by~\eqref{eq:hetero-certificate}.

\begin{lemma}[Differentiability of the optimal log-likelihood]\label{lem:hetero-differentiability}
For every fixed $\Sigma$, the function $\ell^\star_\Sigma:\R^{nd}\to\R$ is locally Lipschitz and hence differentiable almost everywhere.
\end{lemma}
\begin{proof}
Define
\[
F_\Sigma(X):=\ell^\star_\Sigma(X)+\frac{nd}{2}\log(2\pi)+\frac12\sum_i\log\det\Sigma_i+\frac12\sum_i x_i^\top\Sigma_i^{-1}x_i.
\]
The bound $p_{\pi,i}(x_i)\leq(2\pi)^{-d/2}(\det\Sigma_i)^{-1/2}$ gives the upper bound below, and choosing $\pi=\delta_0$ gives the lower bound:
\[
0\leq F_\Sigma(X)\leq\frac12\sum_i x_i^\top\Sigma_i^{-1}x_i.
\]
Expanding the Gaussian exponent gives
\[
F_\Sigma(X)=\sup_{\pi\in\Pcal(\R^d)}\sum_i\log\int\exp\!\left(x_i^\top\Sigma_i^{-1}\theta-\tfrac12\theta^\top\Sigma_i^{-1}\theta\right)\,d\pi(\theta).
\]
H\"older's inequality makes each logarithmic integral convex in $x_i$. Hence $F_\Sigma$ is finite and convex, and thus locally Lipschitz~\cite[Example~9.14]{rockafellarwets1998}. Subtracting the smooth quadratic and applying Rademacher's theorem~\cite[Theorem~9.60]{rockafellarwets1998} proves the claim.
\end{proof}

\begin{lemma}[Equality of likelihood derivatives]\label{lem:hetero-derivatives}
At every dataset $X$ where $\ell^\star_\Sigma$ is differentiable, every $\widehat\pi$ solving~\eqref{eq:hetero-npmle} satisfies
\begin{equation}\label{eq:hetero-derivatives}
\nabla p_{\widehat\pi,i}(x_i)=\widehat p_i\,\nabla_{x_i}\ell^\star_\Sigma(X),\qquad i=1,\ldots,n.
\end{equation}
\end{lemma}
\begin{proof}
By Proposition~\ref{prop:hetero-known}, $\widehat\pi$ is finitely supported, so $X'\mapsto\ell_{X',\Sigma}(\widehat\pi)$ is smooth. Holding $\widehat\pi$ and $\Sigma$ fixed, the function $\ell^\star_\Sigma(X')-\ell_{X',\Sigma}(\widehat\pi)$ is nonnegative and vanishes at $X'=X$. Its gradient there is zero, giving~\eqref{eq:hetero-derivatives}, as in Lemma~\ref{lem:derivatives}.
\end{proof}

\begin{proof}[Proof of Theorem~\ref{thm:hetero-uniqueness}]
Fix $\Sigma$ and a dataset $X$ where $\ell^\star_\Sigma$ is differentiable. Suppose two distinct solutions $\widehat\pi_0,\widehat\pi_1$ exist. By Proposition~\ref{prop:hetero-known}, their combined support consists of finitely many distinct points $\theta_1,\ldots,\theta_m$. Let $u,v$ be their weight vectors, put $w=(u+v)/2$, and define $\Phi_{ij}=\varphi_{\Sigma_i}(x_i-\theta_j)$. Then
\[
w_j>0,\qquad \Phi u=\Phi v=\Phi w=\widehat p,\qquad 0\ne u-v\in\ker\Phi.
\]
By~\eqref{eq:hetero-certificate}, $q_i=(n\widehat p_i)^{-1}$ satisfies $q^\top\Phi=\mathbf1^\top$. In particular, every $h\in\ker\Phi$ satisfies $\sum_jh_j=0$.

Choose $0\ne h\in\ker\Phi$ with inclusion-minimal support $T$: no nonzero vector in $\ker\Phi$ has support strictly contained in $T$. The cancellation argument proving~\eqref{eq:circuit} gives
\[
\ker\Phi_T=\operatorname{span}\{h_T\},\qquad h_j\ne0\ (j\in T),\qquad |T|\geq2.
\]
For sufficiently small $\varepsilon>0$, the weights $w\pm\varepsilon h$ are positive, sum to one, and give likelihood vector $\widehat p$. They therefore define two solutions $\widehat\pi_\pm$. For each $i$, set
\[
g_i(x):=\sum_{j\in T}h_j\varphi_{\Sigma_i}(x-\theta_j)=\frac{p_{\widehat\pi_+,i}(x)-p_{\widehat\pi_-,i}(x)}{2\varepsilon}.
\]
We have $g_i(x_i)=0$ because $\Phi h=0$, and $\nabla g_i(x_i)=0$ by Lemma~\ref{lem:hetero-derivatives}. The Gaussian identity
\[
\Sigma_i\nabla_x\varphi_{\Sigma_i}(x-\theta)=(\theta-x)\varphi_{\Sigma_i}(x-\theta)
\]
therefore gives
\[
0=\Sigma_i\nabla g_i(x_i)+x_i g_i(x_i)=\sum_{j\in T}h_j\Phi_{ij}\theta_j,\qquad i=1,\ldots,n.
\]
Thus, for every coordinate $k$, the vector $(h_j\theta_{j,k})_{j\in T}$ belongs to $\ker\Phi_T=\operatorname{span}\{h_T\}$. Since all $h_j$ on $T$ are nonzero, all $\theta_j$, $j\in T$, coincide, contradicting their distinctness and $|T|\geq2$.

Lemma~\ref{lem:hetero-differentiability} shows that the remaining datasets form a Lebesgue-null set $N_\Sigma$. Finally, Proposition~\ref{prop:hetero-known} supplies a solution with at most $n$ atoms; wherever the solution is unique, it must have this support bound.
\end{proof}

\end{document}